\documentclass[11pt]{amsart}

\usepackage{latexsym,rawfonts}
\usepackage{amsfonts,amssymb}
\usepackage{amsmath,amsthm}

\usepackage[plainpages=false]{hyperref}
\usepackage{graphicx}
\usepackage{color}

\newtheorem{theorem}{Theorem}

\newtheorem{lemma}{Lemma}

\theoremstyle{definition}

\newcommand{\beq}{\begin{equation}}
\newcommand{\eeq}{\end{equation}}
\newcommand{\beqs}{\begin{eqnarray*}}
\newcommand{\eeqs}{\end{eqnarray*}}
\newcommand{\beqn}{\begin{eqnarray}}
\newcommand{\eeqn}{\end{eqnarray}}
\newcommand{\beqa}{\begin{array}}
\newcommand{\eeqa}{\end{array}}

\DeclareMathOperator{\VOL}{vol}

\DeclareMathOperator{\TR}{tr}

\newcommand{\R}{\mathbb{R}}

\newcommand{\dd}{\mathop{}\!\mathrm{d}}
\newcommand{\abs}[1]{\left\vert#1\right\vert}
\newcommand{\set}[1]{\left\{#1\right\}}
\newcommand{\norm}[1]{\left\Vert#1\right\Vert}

\newcommand{\pd}{\partial}
\newcommand{\delbar}{\overline{\nabla}}
\newcommand{\one}{\mathbf{1}}

\newcommand{\uS}{\mathbb{S}^{n-1}}

\newcommand{\MA}{Monge-Amp\`ere }
\newcommand{\OM}{Orlicz-Minkowski }

\begin{document}

\title{Existence of smooth even solutions to the dual Orlicz-Minkowski problem}

\author{Li Chen}
\address{Faculty of Mathematics and Statistics, Hubei Key Laboratory of Applied Mathematics, Hubei University, Wuhan 430062, P.R. China}
\email{chernli@163.com}

\author{YanNan Liu}
\address{School of Mathematics and Statistics, Beijing Technology and Business University, Beijing 100048, P.R. China}
\email{liuyn@th.btbu.edu.cn}

\author{Jian Lu}
\address{South China Research Center for Applied Mathematics and Interdisciplinary Studies, South China Normal University, Guangzhou 510631, P.R. China} 
\email{jianlu@m.scnu.edu.cn; \  lj-tshu04@163.com}

\author{Ni Xiang}
\address{Faculty of Mathematics and Statistics, Hubei Key Laboratory of Applied Mathematics, Hubei University, Wuhan 430062, P.R. China}
\email{nixiang@hubu.edu.cn}

\thanks{This research was supported by funds from Natural Science Foundation of China No.11871432 and No.11971157, Hubei Provincial Department of Education Key Projects D20171004, D20181003. The second author was also supported in part by Beijing Natural Science Foundation No.1172005.}

\date{}

\begin{abstract}
In this paper we study the dual Orlicz-Minkowski problem, which is a
generalization of the dual Minkowski problem in convex geometry.
By considering a geometric flow involving Gauss curvature and functions of
normal vectors and radial vectors, we obtain a new existence result of smooth
even solutions to this problem for smooth even measures. 
\end{abstract}

\keywords{
  Monge-Amp\`ere equation,
  dual Orlicz-Minkowski problem,
  Gauss curvature flow,
  Existence of solutions.
}

\subjclass[2010]{
35J96, 52A20, 53C44.
}

\maketitle
\vskip4ex

\section{Introduction}

Let $\varphi : (0,+\infty) \to (0,+\infty)$ and $G : \R^n\backslash 0 \to
(0,+\infty)$ be two continuous positive functions.
For a given positive function $f$ defined on the unit sphere $\uS$, we are
concerned with the solvability of the \MA type equation
\begin{equation} \label{dOMP-f}
  c\, \varphi(h) G(\delbar h) \det(\nabla^2h +hI) =f \text{ on } \uS
\end{equation}
for some positive constant $c$ and some support function $h$ of a
bounded convex body $K$ in the Euclidean space $\R^n$ containing the origin.
Here $\nabla$ is the covariant derivative with respect to an orthonormal frame
on $\uS$, $I$ is the unit matrix of order $n-1$, and
$\delbar h(x) =\nabla h(x) +h(x) x$ is the point on $\partial K$ whose
outer unit normal vector is $x\in\uS$.

The simplest case of Eq. \eqref{dOMP-f} is when $\varphi$ and $G$ are both
constant functions, which is the \emph{classical Minkowski problem}.
A typical example of Eq. \eqref{dOMP-f} is when $\varphi(s)=s^{1-p}$ and
$G(y)=|y|^{q-n}$, which is the \emph{$L_p$ dual Minkowski problem}.
The $L_p$ dual Minkowski problem was recently introduced in
\cite{LYZ.Adv.329-2018.85}, which unifies the dual Minkowski problem ($p=0$) and
the $L_p$-Minkowski problem ($q=n$). 
The dual Minkowski problem was first proposed by Huang, Lutwak, Yang and Zhang
in their recent groundbreaking work \cite{HLYZ.Acta.216-2016.325} and then
followed by 
\cite{BHP.JDG.109-2018.411,
  HP.Adv.323-2018.114,
  HJ.JFA.277-2019.2209, 
  LSW.JEMSJ.22-2020.893,
  Zha.CVPDE.56-2017.18,
  Zha.JDG.110-2018.543}.
The $L_p$-Minkowski problem was introduced by Lutwak \cite{Lut.JDG.38-1993.131}
in 1993 and has been extensively studied since then; see e.g. 
\cite{
  BLYZ.JAMS.26-2013.831,
  CLZ.TAMS.371-2019.2623,
  Zhu.Adv.262-2014.909}
for the logarithmic Minkowski problem ($p=0$),
\cite{JLW.JFA.274-2018.826, 
  JLZ.CVPDE.55-2016.41,
  Lu.SCM.61-2018.511,
  Lu.JDE.266-2019.4394,
  LW.JDE.254-2013.983,
  Zhu.JDG.101-2015.159}
for the centroaffine Minkowski problem ($p=-n$), and
\cite{CW.Adv.205-2006.33,
  HLYZ.DCG.33-2005.699,
  LYZ.TAMS.356-2004.4359,
  Sta.Adv.167-2002.160}
for other cases of the $L_p$-Minkowski problem.
For the general $L_p$ dual Minkowski problem, much progress has already been made 
\cite{BF.JDE.266-2019.7980,
  CHZ.MA.373-2019.953,
  CCL,
  HLYZ.JDG.110-2018.1,
  HZ.Adv.332-2018.57,
  LLL}.
Another important special case of Eq. \eqref{dOMP-f} is when $G$ is a constant
function, which is the \emph{\OM problem}, see 
\cite{BBC.AiAM.111-2019.101937,
  HLYZ.Adv.224-2010.2485,
  HH.DCG.48-2012.281,
  JL.Adv.344-2019.262}.

Equation \eqref{dOMP-f} is in fact the \emph{dual \OM problem}, which 
is a basic problem in the dual Orlicz-Brunn-Minkowski theory in modern
convex geometry.
This problem is a generalization of the above-mentioned Minkowski type problems,
and it asks what are the necessary and sufficient conditions for a Borel measure
on the unit sphere $\uS$ to be a multiple of the Orlicz dual curvature measure
of a convex body in $\R^n$ \cite{GHW+.CVPDE.58-2019.12}. 
When the given measure is absolutely continuous, it is equivalent to solving Eq.
\eqref{dOMP-f}.

Some existence results of solutions to Eq. \eqref{dOMP-f} have already been known.
In \cite{GHW+.CVPDE.58-2019.12, GHXY.CVPDE.59-2020.15},
several sufficient conditions for the existence of generalized solutions
were obtained for general measures or even measures.
In \cite{LL.TAMS.373-2020.5833}, an existence result of smooth solutions was proved for smooth
measures, in which the constant $c$ can be taken as $1$.

In this paper we study the existence of smooth even solutions to Eq.
\eqref{dOMP-f}.
A function $g:\uS\to\R$ is called \emph{even} if
\begin{equation*}
g(-x)=g(x), \quad \forall\,x\in\uS.
\end{equation*}
Note that a support function is even if and only if the convex body determined
by this support function is origin-symmetric.
Let $B_1$ be the unit ball in $\R^n$.
We obtain the following existence result for the even case of the dual \OM problem.

\begin{theorem}\label{thm1}
Suppose $\varphi\in C^\infty(0,+\infty)$ is a positive function, and $G\in
C^\infty(\R^n\backslash0)$ and $f\in C^\infty(\uS)$ are two positive and even
functions. 
If $\varphi$ and $G$ satisfy either 
\begin{equation}\label{cond1}
  \int_0^1 \frac{\dd s}{\varphi(s)}<+\infty, \quad
  \int_1^{\infty} \frac{\dd s}{\varphi(s)}=+\infty, \quad\text{and }
  \int_{B_1} G(y) \dd y <+\infty;
\end{equation}
or
\begin{equation}\label{cond2}
  \int_1^{\infty} \frac{\dd s}{\varphi(s)}<+\infty, \quad
  \int_{B_1} G(y) \dd y =+\infty, \quad\text{and }
  \int_{\R^n\backslash B_1} G(y) \dd y <+\infty,
\end{equation}
then there exists a smooth even solution to equation \eqref{dOMP-f} for some
positive constant $c$. 
\end{theorem}

To the best of our knowledge, both conditions \eqref{cond1} and \eqref{cond2}
are new in the study of the dual \OM problem. 
We note that the special case when $\varphi(s)=s^{1-p}$ and
$G(y)=|y|^{q-n}$ satisfies the condition \eqref{cond1} if $p>0$ and $q>0$, and
satisfies the condition \eqref{cond2} if $p<0$ and $q<0$. 
So Theorem \ref{thm1} includes the even $L_p$ dual Minkowski problem with
$pq>0$, which was recently studied by Chen, Huang and Zhao
\cite{CHZ.MA.373-2019.953}.

In fact, our proof of Theorem \ref{thm1} is inspired by
\cite{CHZ.MA.373-2019.953},
where the existence of smooth even solutions to the $L_p$ dual Minkowski problem
was obtained by studying a generalized Gauss curvature flow.
Let $M_{0}$ be a smooth, closed, uniformly convex hypersurface in the Euclidean
space $\R^{n}$, which is origin-symmetric and given by a smooth embedding
$X_{0}: \mathbb{S}^{n-1} \rightarrow \R^{n}$. 
In this paper we consider a family of closed hypersurfaces $\set{M_t}$ given by
$M_t=X(\uS,t)$, where $X: \mathbb{S}^{n-1}\times[0,T) \rightarrow \R^{n}$ is a
smooth map satisfying the following initial value problem:
\begin{equation}\label{flow}
  \begin{split}
    \frac{\pd X}{\pd t} (x,t)
    &= -f(\nu)\eta(t) \frac{\langle X, \nu \rangle}{\varphi(\langle X, \nu \rangle) G(X)} \mathcal{K} \nu +  X,\\
    X(x,0) &= X_{0}(x).
  \end{split}
\end{equation}
Here $\nu$ is the unit outer normal vector of the hypersurface $M_t$ at the
point $X(x,t)$, $\eta$ is a scalar function to be determined in order to keep
$M_t$ normalized in a certain sense, $\langle \cdot,\cdot \rangle$ is the
standard inner product in $\R^n$, $\mathcal{K}$ is the Gauss curvature of
$M_{t}$ at $X(x,t)$, and $T$ is the maximal time for which the solution exists. 
The Gauss curvature flow and its various generalizations have been extensively
studied by many scholars; see for example 
\cite{And.IMRN.1997.1001,
AGN.Adv.299-2016.174,
BCD.Acta.219-2017.1,
BIS.AP.12-2019.259,
CW.AIHPANL.17-2000.733,
GL.DMJ.75-1994.79,
Ger.CVPDE.49-2014.471,
Ham.CAG.2-1994.155,
Iva.JFA.271-2016.2133,
LSW.JEMSJ.22-2020.893,
LL.TAMS.373-2020.5833,
Urb.JDG.33-1991.91}
and the references therein.

We will prove that the flow \eqref{flow} has a long-time solution, which
sub-converges to a solution to Eq. \eqref{dOMP-f} as time tends to infinity. 
To prove the long-time existence, the key step is
to obtain uniformly positive lower and upper bounds for
the support function of $M_t$.
Since flow \eqref{flow} involves two inhomogeneous functions 
$\varphi$ and $G$, which is more complicated than that of
\cite{CHZ.MA.373-2019.953},
we need more efforts to obtain these estimates, including constructing a proper $\eta(t)$
and imposing suitable constraints on $\varphi$ and $G$; see section 3 for details.

In this paper we obtain the long-time existence and convergence of the flow \eqref{flow}.

\begin{theorem}\label{thm2}
Suppose $\varphi$, $G$ and $f$ satisfy the assumptions of Theorem \ref{thm1},
and $\eta(t)$ is given by \eqref{eta} in section 3.
For any smooth, closed, origin-symmetric, uniformly convex hypersurface
$M_{0}$ in $\R^{n}$,
the flow \eqref{flow} has a unique smooth solution, which exists for all time
$t>0$.
Moreover, when $t\to\infty$, a subsequence of $M_{t}=X(\uS,t)$ converges in
$C^{\infty}$ to a smooth, closed, origin-symmetric, uniformly convex
hypersurface, whose support function is a smooth even solution to Eq.
\eqref{dOMP-f} for some positive constant $c$.
\end{theorem}

This paper is organized as follows.
In section 2, we give some basic knowledge about convex hypersurfaces and the
flow \eqref{flow}. 
In section 3, more properties of the flow \eqref{flow} will be proved, based on
which we can obtain the uniform lower and upper bounds of support functions
of $\set{M_t}$ via delicate analyses.
Then the bounds of principal curvatures and the long-time existence will be
proved in section 4.
In the last section, we prove the sub-convergence of $M_t$ as $t$ tends to infinity,
completing the proofs of our theorems.

\section{Preliminaries}

\subsection{Basic properties of convex hypersurfaces}

We first recall some basic properties of convex hypersurfaces in
$\R^n$; see \cite{Urb.JDG.33-1991.91} for details.
Let $M$ be a smooth, closed, uniformly convex hypersurface in $\R^n$ enclosing
the origin.
The support function $h$ of $M$ is defined as
\begin{equation} \label{h0}
h(x) := \max_{y\in M} \langle y,x \rangle, \quad \forall\, x\in\uS,
\end{equation}
where $\langle \cdot,\cdot \rangle$ is the standard inner product in $\R^n$.

The convex hypersurface $M$ can be recovered by its support function $h$.
In fact, writing the Gauss map of $M$ as $\nu_M$, we parametrize $M$ by
$X : \uS\to M$ which is given as
\begin{equation*}
X(x) =\nu_M^{-1}(x), \quad \forall\,x\in\uS.
\end{equation*}
Note that $x$ is the unit outer normal vector of $M$ at $X(x)$.
On the other hand, one can easily check that the maximum in the definition
\eqref{h0} is attained at $y=\nu_M^{-1}(x)$, namely
\begin{equation} \label{h}
  h(x) = \langle x, X(x)\rangle, \quad\forall\, x \in \mathbb{S}^{n-1}. 
\end{equation}
Let $e_{ij}$ be the standard metric of the unit sphere $\mathbb{S}^{n-1}$, and
$\nabla$ be the corresponding connection on $\mathbb{S}^{n-1}$.
Differentiating the both sides of \eqref{h}, we have
\begin{equation*}
  \nabla_{i} h = \langle \nabla_{i}x, X(x)\rangle + \langle x, \nabla_{i}X(x)\rangle. 
\end{equation*}
Since $\nabla_{i}X(x)$ is tangent to $M$ at $X(x)$, there is
\begin{equation*}
  \nabla_{i} h = \langle \nabla_{i}x, X(x)\rangle,
\end{equation*}
which together with \eqref{h} implies that
\begin{equation}\label{Xh}
  X(x) = \nabla h(x) + h(x)x, \quad \forall\,x\in\uS.
\end{equation}

By differentiating \eqref{h} twice, the second fundamental form $A_{ij}$ of $M$
can be also computed in terms of the support function:
\begin{equation} \label{A}
  A_{ij} =  \nabla_{ij}h + he_{ij}, 
\end{equation}
where $\nabla_{ij}$ denotes the second order covariant derivative with respect to $e_{ij}$.
The induced metric matrix $g_{ij}$ of $M$ can be derived by Weingarten's formula:
\begin{equation} \label{g}
  e_{ij} = \langle \nabla_{i}x, \nabla_{j}x\rangle  = A_{ik}A_{lj}g^{kl}. 
\end{equation}
The principal radii of curvature are eigenvalues of the matrix $b_{ij} =
A^{ik}g_{jk}$.
When considering a smooth local orthonormal frame on $\uS$, by virtue of
\eqref{A} and \eqref{g}, we have 
\begin{equation} \label{bij}
  b_{ij} = A_{ij} = \nabla_{ij}h + h\delta_{ij}.
\end{equation}
Now the Gauss curvature of $M$ at $X(x)$ is given by
\begin{equation*}
\mathcal{K}(x) = [\det(\nabla_{ij}h + h\delta_{ij})]^{-1}. 
\end{equation*}
We shall use $b^{ij}$ to denote the inverse matrix of $b_{ij}$.

The radial function $\rho$ of the convex hypersurface $M$ is defined as
\begin{equation*}
\rho(u) :=\max\set{\lambda>0 : \lambda u\in M}, \quad\forall\, u\in\uS.
\end{equation*}
Note that $\rho(u)u\in M$.
The Gauss map $\nu_M$ can be computed as 
\begin{equation*}
  \nu_M(\rho(u)u) = \frac{\rho(u)u - \nabla \rho}{\sqrt{\rho^{2}+|\nabla \rho|^{2}}}.
\end{equation*}
If we connect $u$ and $x$ through the following equality:
\begin{equation}\label{eq:8}
\rho(u)u =X(x) =\nabla h(x) +h(x)x =\delbar h(x),
\end{equation}
then there is
\begin{equation}\label{eq:9}
x= \frac{\rho(u)u - \nabla \rho}{\sqrt{\rho^{2}+|\nabla \rho|^{2}}}.
\end{equation}

\subsection{Geometric flow}

Recalling the evolution equation of $X(x,t)$ in the geometric flow \eqref{flow},
and using similar computations as in \cite{Urb.JDG.33-1991.91}, we obtain the
evolution equation of the corresponding support function $h(x,t)$:
\begin{equation}\label{ht}
  \frac{\pd h}{\pd t} (x,t)
  = -f(x)\eta(t) \frac{h(x,t)}{\varphi(h) G(\delbar h)} \mathcal{K}(x,t) +  h(x,t)
  \ \text{ in }\ \uS\times(0,T).
\end{equation}
Since $M_t$ can be recovered by $h(\cdot,t)$, the flow \eqref{ht} is equivalent
to the original flow \eqref{flow}.

Denote the radial function of $M_t$ by $\rho(u,t)$.
For each $t$, 
let $u$ and $x$ be related through the following equality:
\begin{equation*}
\rho(u,t)u =\delbar h(x,t) =\nabla h(x,t) +h(x,t)x.
\end{equation*}
Therefore, $x$ can be expressed as $x=x(u,t)$.
Taking the inner product with $x$ on both sides, we have
\begin{equation}\label{eq:2}
\rho(u,t) \langle x,u \rangle =h(x,t).
\end{equation}
Differentiating it with respect to $t$, we obtain
\begin{equation}\label{eq:1}
  \pd_t\rho(u,t) \langle x,u \rangle 
  + \rho(u,t) \langle \pd_tx,u \rangle
  = \langle \nabla h(x,t), \pd_tx \rangle
  +\pd_th(x,t).
\end{equation}
Note that
\begin{equation*}
  \begin{split}
    \rho(u,t) \langle \pd_tx,u \rangle - \langle \nabla h(x,t), \pd_tx \rangle
    &= \langle \pd_tx,\rho(u,t)u -\nabla h(x,t) \rangle \\
    &= \langle \pd_tx,h(x,t)x \rangle \\
    &=0.
  \end{split}
\end{equation*}
Combining it with \eqref{eq:1} and recalling \eqref{eq:2}, we have
\begin{equation}\label{eq:4}
  \pd_t\rho(u,t) =\frac{\pd_th(x,t)}{\langle x,u \rangle}
  =\frac{\rho(u,t)}{h(x,t)}\pd_th(x,t).
\end{equation}
Now by virtue of \eqref{ht}, we obtain the evolution equation of $\rho(u,t)$:
\begin{equation}\label{rhot}
  \frac{\pd \rho}{\pd t} (u,t)
  = -f(x)\eta(t) \frac{\rho(u,t)}{\varphi(h) G(\rho u)} \mathcal{K}(x,t) +  \rho(u,t)
  \ \text{ in }\ \uS\times(0,T),
\end{equation}
where $x=x(u,t)$ is the unit outer normal vector of $M_t$ at the point $\rho(u,t)u$.

\section{Uniform bounds of support functions}

In this section, we will derive uniformly positive lower and upper bounds of
support functions along the flow \eqref{flow}.

For convenience, in the following of this paper, we always assume that $M_{0}$
is a smooth, closed, origin-symmetric, uniformly convex hypersurface in
$\R^{n}$, $\varphi$, $G$ and $f$ are functions satisfying the assumptions of
Theorem \ref{thm1}, and $h: \uS\times[0,T)\to \R$ is a smooth solution to the
evolution equation \eqref{ht} with the initial $h(\cdot,0)$ being the support function
of $M_0$. 
Here $T$ is the maximal time for which the solution exists.
Let $M_t$ be the convex hypersurface determined by $h(\cdot,t)$, 
$\rho(\cdot,t)$ be the corresponding radial function,
and $K_t$ be the convex body enclosed by $M_t$.

To obtain uniform bounds of $h(\cdot,t)$, we need to choose an appropriate
scalar function $\eta(t)$.
Note that Eq. \eqref{dOMP-f} admits a variational structure which involves the
general dual volume of convex bodies. 
For any convex body $K$ containing the origin in its interiors, we define its
\emph{$G$ dual volume} as 
\begin{equation} 
  \widetilde{V}_G(K)=
  \begin{cases}
    \displaystyle\int_K G(y)\dd y, & \text{when \eqref{cond1} holds}, \\[2.5ex]
    \displaystyle\int_{\R^n\backslash K} G(y)\dd y, & \text{when \eqref{cond2} holds}.
  \end{cases}
  \label{Gdv}
\end{equation}
By virtue of the third inequalities in the conditions \eqref{cond1} or
\eqref{cond2}, $\widetilde{V}_G(K)$ is well defined.
The concept of $G$ dual volume of a convex body is a natural extension of $q$-th
dual volume; see \cite[Page 8]{GHW+.CVPDE.58-2019.12} for more discussions.
We want to keep the $G$ dual volumes of $\set{K_t}$ unchanged along the flow \eqref{ht}.
In fact, we have the following lemma.

\begin{lemma}\label{lem01}
When $\eta(t)$ is given as
\begin{equation} \label{eta}
  \eta(t) 
  = \frac{\displaystyle\int_{\uS} G(\delbar h) h(x,t) /\mathcal{K} \,\dd x}
  {\displaystyle\int_{\uS} f(x) h(x,t)/ \varphi(h) \dd x},
\end{equation}
the $G$ dual volumes of $\set{K_t}$ remain unchanged, namely
\begin{equation*}
  \widetilde{V}_G(K_t)= \widetilde{V}_G(K_0), \quad \forall\, t\in[0,T). 
\end{equation*}
\end{lemma}

\begin{proof}
Observing that
\begin{equation*}
  \widetilde{V}_G(K_t)=
  \begin{cases}
    \displaystyle\int_{\uS} \dd u \int_0^{\rho(u,t)} G(r u) r^{n-1} \dd r,
    & \text{when \eqref{cond1} holds}, \\[2.5ex]
    \displaystyle\int_{\uS} \dd u \int_{\rho(u,t)}^{\infty} G(r u) r^{n-1} \dd r,
    & \text{when \eqref{cond2} holds},
  \end{cases}
\end{equation*}
we have
\begin{equation*} 
  \begin{split}
    \frac{\dd}{\dd t} \widetilde{V}_G(K_t)
    &= \pm\int_{\uS} G(\rho u) \rho(u,t)^{n-1} \pd_t \rho(u,t) \dd u \\
    &= \pm\int_{\uS} G(\rho u) \rho(u,t)^{n} \frac{\pd_t h(x,t)}{h(x,t)} \dd u \\
    &= \pm\int_{\uS} G(\delbar h)  \frac{\pd_t h(x,t)}{\mathcal{K}(x,t)} \dd x,
  \end{split}
\end{equation*}
where the second equality is due to \eqref{eq:4}, and the third equality is the
result of integration by substitution. 
Recalling the evolution equation \eqref{ht}, namely
\begin{equation*}
  \frac{G(\delbar h)}{\mathcal{K}(x,t)} \pd_th(x,t)
  = -f(x)\eta(t) \frac{h(x,t)}{\varphi(h)}  + \frac{G(\delbar h)}{\mathcal{K}(x,t)} h(x,t),
\end{equation*}
we obtain
\begin{equation*} 
  \frac{\dd}{\dd t} \widetilde{V}_G(K_t)
  = \pm\left[-\eta(t) \int_{\uS} f(x) \frac{h(x,t)}{\varphi(h)} \dd x
    + \int_{\uS} \frac{G(\delbar h)}{\mathcal{K}(x,t)} h(x,t) \dd x\right].
\end{equation*}
Therefore, if $\eta(t)$ is given by \eqref{eta}, there is
\begin{equation*}
  \frac{\dd}{\dd t} \widetilde{V}_G(K_t) \equiv0,
\end{equation*}
which completes the proof of this lemma. 
\end{proof}

The another important advantage of the choice of $\eta(t)$ is that it can ensure
the monotonicity of a related functional. 
For simplicity, write
\begin{equation} 
  \phi(s)=
  \begin{cases}
    \displaystyle\int_0^s 1/\varphi(\tau) \dd\tau, & \text{when \eqref{cond1} holds}, \\[2.5ex]
    \displaystyle\int_s^\infty 1/\varphi(\tau) \dd\tau, & \text{when \eqref{cond2} holds}.
  \end{cases}
  \label{phi}
\end{equation}
By virtue of the first inequalities in the conditions \eqref{cond1} or
\eqref{cond2}, $\phi(s)$ is well defined for every $s\in(0,+\infty)$.
Consider the following functional:
\begin{equation}\label{Jt}
J(t)=\int_{\uS} \phi(h(x,t)) f(x) \dd x, \quad t\geq 0,
\end{equation}  
which will turn out to be monotonic along the flow \eqref{ht}.

\begin{lemma}\label{lem02}
When the condition \eqref{cond1} holds, $J(t)$ is non-increasing. 
When the condition \eqref{cond2} holds, $J(t)$ is non-decreasing. 
\end{lemma}

\begin{proof}
We first prove for the case when \eqref{cond1} holds.
Now $\phi'(s)=1/\varphi(s)$, and
\begin{equation*}
  J'(t)
  =\int_{\uS} \phi'(h) \pd_th(x,t) f(x) \dd x
  =\int_{\uS} \frac{f(x)}{\varphi(h)} \pd_th(x,t) \dd x.
\end{equation*}  
Recalling \eqref{ht}, namely
\begin{equation*}
  \frac{f(x)}{\varphi(h)} \pd_th(x,t)
  = -\eta(t)\frac{f(x)^2h}{\varphi(h)^2G(\delbar h)} \mathcal{K}(x,t)
  + \frac{f(x)h}{\varphi(h)},
\end{equation*}
we have
\begin{equation*}
  J'(t)
  =
  \int_{\uS}\frac{fh}{\varphi(h)} \dd x
  -\eta(t) \int_{\uS}\frac{f^2h\mathcal{K}}{\varphi(h)^2G(\delbar h)} \dd x.
\end{equation*}
By the definition of $\eta(t)$ in \eqref{eta}, namely
\begin{equation*} 
  \eta(t) \int_{\uS} f h/ \varphi(h) \dd x
  = \int_{\uS} G(\delbar h) h /\mathcal{K} \,\dd x,
\end{equation*}
there is
\begin{equation}\label{eq:23}
  \begin{split}
    J'&(t) \int_{\uS} f h/ \varphi(h) \dd x \\
    &= \left( \int_{\uS}\frac{fh}{\varphi(h)} \dd x \right)^2
    -\int_{\uS} \frac{G(\delbar h) h}{\mathcal{K}} \dd x
    \int_{\uS}\frac{f^2h\mathcal{K}}{\varphi(h)^2G(\delbar h)} \dd x \\
    &= \left(
      \int_{\uS}
      \sqrt{\frac{Gh}{\mathcal{K}}}
      \cdot
      \frac{f}{\varphi}
      \sqrt{\frac{h\mathcal{K}}{G}}
      \dd x
    \right)^2
    -\int_{\uS} \frac{G h}{\mathcal{K}} \dd x
    \int_{\uS}\frac{f^2h\mathcal{K}}{\varphi^2G} \dd x \\
    &\leq0,
  \end{split}
\end{equation}
where the last inequality is due to the H\"older's inequality.
Therefore, $J(t)$ is non-increasing.

When the condition \eqref{cond2} holds, the proof is almost the same.
Since $\phi'(s)=-1/\varphi(s)$, there is
\begin{equation*}
  \begin{split}
    J'(t)
    &=\int_{\uS} \phi'(h) \pd_th(x,t) f(x) \dd x \\
    &=-\int_{\uS} \frac{f(x)}{\varphi(h)} \pd_th(x,t) \dd x \\
    &\geq0,
  \end{split}
\end{equation*}
where the last inequality is true due to the above argument.
Therefore, $J(t)$ is non-decreasing. 
\end{proof}

After these preparations, we now derive the uniform bounds of $h(\cdot,t)$.

\begin{lemma}\label{lem03}
There exists a positive constant $C$
independent of $t$, such that for every $t\in[0,T)$,
\begin{equation}\label{h1}
1/C \leq h(\cdot,t) \leq C \text{ on }\uS.
\end{equation}
It means that
\begin{equation}\label{rho1}
1/C \leq \rho(\cdot,t) \leq C \text{ on }\uS.
\end{equation} 
\end{lemma}

\begin{proof}
(a)
We first prove for the case when the condition \eqref{cond1} holds.
For each $t$, write
\begin{equation*}
R_t = \max_{x\in\uS} h(x,t) = h(x_t,t)
\end{equation*}
for some $x_t\in\uS$.
Since $G$ and $f$ are even, by the origin-symmetry of $M_0$,
$M_{t}$ is origin-symmetric.
Then we have by the definition of support function
that
\begin{equation*}
h(x,t)\geq R_t\cdot|\langle x,x_t \rangle|, \quad \forall x\in\uS.
\end{equation*}
By Lemma \ref{lem02}, $J(t)$ is non-increasing. 
Recalling the definitions of $\phi$ and $J$ in \eqref{phi} and \eqref{Jt}
respectively, we have the following estimates: 
\begin{equation*}
\begin{split}
  J(0) &\geq J(t) \\
  &\geq f_{\min} \int_{\uS} \phi(h(x,t)) \dd x \\
  &\geq f_{\min} \int_{\uS} \phi(R_t\cdot|\langle x,x_t \rangle|) \dd x \\
  &= f_{\min} \int_{\uS} \phi(R_t\cdot|x_1|) \dd x.
\end{split}
\end{equation*}
Denote $S_1 =\set{x\in\uS : |x_1|\geq 1/2}$, then
\begin{equation*}
\begin{split}
  J(0) 
  &\geq f_{\min} \int_{S_1} \phi(R_t\cdot|x_1|) \dd x \\
  &\geq f_{\min} \int_{S_1} \phi(R_t/2) \dd x \\
  &= f_{\min} \phi(R_t/2) |S_1|,
\end{split}
\end{equation*}
which implies that $\phi(R_t/2)$ is uniformly bounded from above.
By its definition, $\phi(s)$ is strictly increasing.
Moreover, by the second equality of the condition \eqref{cond1}, $\phi(s)$ tends to
$+\infty$ as $s\to+\infty$. 
Thus $R_t$ is uniformly bounded from above, leading to the second inequality in
\eqref{h1}.

By the second inequality in \eqref{h1}, there is
\begin{equation*}
K_t\subset B_C, \quad \forall\,t\in[0,T),
\end{equation*}
where $B_C$ denotes the ball centered at the origin with radius $C$.
By virtue of the third inequality of the condition \eqref{cond1},
\begin{equation*}
\int_{B_C} G(y)\dd y<+\infty,
\end{equation*}
which implies that there exists a $\delta>0$ such that
\begin{equation*}
  \int_A G(y)\dd y < \widetilde{V}_G(K_0)
  \quad 
  \text{for every measurable set } A\subset B_C \mbox{ with } |A|<\delta.
\end{equation*}
Recalling Lemma \ref{lem01}, $\widetilde{V}_G(K_t) = \widetilde{V}_G(K_0)$, namely
\begin{equation*}
  \int_{K_t} G(y)\dd y = \widetilde{V}_G(K_0),
\end{equation*}
which implies that
\begin{equation} \label{eq:3}
  \VOL(K_t)\geq\delta, \quad \forall\,t\in[0,T).
\end{equation}

Now one can easily obtain the
uniform positive lower bound of $h(\cdot,t)$.
In fact, recalling the concept of minimum ellipsoid of an origin-symmetric convex body,
there exists a positive constant $C_n$ depending only on $n$, such that 
\begin{equation*}
\VOL(K_t) \leq C_n R_t^{n-1} \cdot \min_{x\in\uS} h(x,t).
\end{equation*}
Therefore the uniform positive lower bound of $h(\cdot,t)$ follows from their uniform
upper bound and the uniform volume estimate \eqref{eq:3}.

(b)
Now we consider the case when the condition \eqref{cond2} holds.
We begin with the following estimate about $\widetilde{V}_G(K_t)$:
\begin{equation}\label{eq:5}
  \begin{split}
    \widetilde{V}_G(K_t)
    &= \int_{\R^n\backslash K_t} G(y) \dd y \\ 
    &\geq \int_{B_1\backslash K_t} G(y) \dd y \\ 
    &= \int_{B_1} G(y) \one_{B_1\backslash K_t}(y) \dd y,
  \end{split}
\end{equation}
where $\one$ denotes the characteristic function.
Let us now prove the first inequality of \eqref{h1}.
Suppose to the contrary that there exists a sequence of times $t_k\in[0,T)$ such
that
\begin{equation*}
  \min_{\uS} h(\cdot, t_k) \to0^+ \text{ as }k\to\infty.
\end{equation*}
Let $L_k=B_1\cap K_{t_k}$.
Then each $L_k$ is an origin-symmetric convex body contained in the unit ball
$B_1$. By the Blaschke selection theorem, $\set{L_k}$ has a subsequence which
converges to a nonempty, compact, convex subset. Without loss of generality, we
assume 
\begin{equation*}
L_k\to L \text{ as }k\to\infty,
\end{equation*}
where $L$ is an origin-symmetric convex subset of $B_1$.
Denote the support functions of $L_k$ and $L$ by $h_{L_k}$ and $h_L$
respectively. We have 
\begin{equation*}
  \begin{split}
    \min_{\uS} h_L
    &= \lim_{k\to\infty} \min_{\uS} h_{L_k} \\
    &\leq \lim_{k\to\infty} \min_{\uS} h(\cdot,t_k) \\
    &= 0,
  \end{split}
\end{equation*}
which together with that $L$ is origin-symmetric implies that $L$ is contained
in a hyperplane in $\R^n$.
Therefore, the characteristic function of $L$ is zero out of this hyperplane,
namely
\begin{equation*}
  \one_{L}=0 \ \text{ a.e. in } B_1.
\end{equation*}
Thus, for a.e. $y\in B_1$, there is
\begin{equation*}
  \begin{split}
   \lim_{k\to\infty} \one_{B_1\backslash K_{t_k}}(y)
    &= \lim_{k\to\infty} [1-\one_{L_k}(y)] \\
    &= 1-\one_{L}(y) \\
    &=1.
  \end{split}
\end{equation*}
Recalling the estimate \eqref{eq:5}, we have by the Fatou's lemma that
\begin{equation*}
  \begin{split}
    \liminf_{k\to\infty} \widetilde{V}_G(K_{t_k})
    &\geq \liminf_{k\to\infty} \int_{B_1} G(y) \one_{B_1\backslash K_{t_k}}(y) \dd y \\
    &\geq \int_{B_1} \liminf_{k\to\infty} G(y) \one_{B_1\backslash K_{t_k}}(y) \dd y \\
    &= \int_{B_1} G(y) \dd y \\
    &= +\infty,
  \end{split}
\end{equation*}
where the last equality is due to the second equality of the condition
\eqref{cond2}.
However, by Lemma \ref{lem01}, $\widetilde{V}_G(K_{t_k})= \widetilde{V}_G(K_0)$,
which means that the above inequality is impossible to hold.
Hence, the first inequality of \eqref{h1} is true.

Denote the uniformly positive lower bound of $h(\cdot,t)$ by $C_1$.
Let $\delta$ be a small positive constant satisfying the following inequality:
\begin{equation}\label{eq:6}
  \abs{\set{x\in\uS : |x_1|<\delta}}
  \leq
  \frac{J(0)}{2f_{\max}\phi(C_1)}.
\end{equation}
Again we write for each $t$ that
\begin{equation*}
R_t = \max_{x\in\uS} h(x,t) = h(x_t,t)
\end{equation*}
for some $x_t\in\uS$, and we have
\begin{equation*}
h(x,t)\geq R_t\cdot|\langle x,x_t \rangle|, \quad \forall x\in\uS.
\end{equation*}
By Lemma \ref{lem02}, $J(t)$ is non-decreasing.
By the definition of $\phi$ in \eqref{phi}, $\phi$ is strictly decreasing.
Denoting
\begin{equation*}
S_t =\set{x\in\uS : |\langle x,x_t \rangle|<\delta},
\end{equation*}
and recalling the expression of $J(t)$ in \eqref{Jt},
we have
\begin{equation}\label{eq:7}
  \begin{split}
    J(0) &\leq J(t) \\
    &= \int_{\uS\backslash S_t} \phi(h(x,t)) f(x) \dd x 
    + \int_{S_t} \phi(h(x,t)) f(x) \dd x \\
    &\leq f_{\max} \int_{\uS\backslash S_t} \phi(R_t\cdot|\langle x,x_t \rangle|) \dd x 
    + f_{\max} \int_{S_t} \phi(C_1) \dd x \\
    &\leq f_{\max} \int_{\uS\backslash S_t} \phi(R_t\delta) \dd x 
    + f_{\max} \phi(C_1) |S_t| \\
    &= f_{\max} \phi(R_t\delta) |\uS\backslash S_t|
    + f_{\max} \phi(C_1) |S_t|.
  \end{split}
\end{equation}
Recalling the definition of $\delta$ in \eqref{eq:6}, one can see that
\begin{equation*}
  |S_t|=
  \abs{\set{x\in\uS : |x_1|<\delta}}
  \leq
  \frac{J(0)}{2f_{\max}\phi(C_1)},
\end{equation*}
namely
\begin{equation*}
  f_{\max}\phi(C_1) |S_t|
  \leq
  \frac{J(0)}{2}.
\end{equation*}
Inserting it into \eqref{eq:7}, we obtain
\begin{equation*}
  \begin{split}
    \frac{J(0)}{2}
    &\leq f_{\max} \phi(R_t\delta) |\uS\backslash S_t| \\
    &\leq C_n f_{\max} \phi(R_t\delta),
  \end{split}
\end{equation*}
which implies that $\phi(R_t\delta)$ has a uniformly positive lower bound.
By its definition, $\phi(s)$ is strictly decreasing, and tends to $0^+$ as
$s\to+\infty$.
Thus $R_t$ is uniformly bounded from above, leading to the second inequality in
\eqref{h1}. 
The proof of this lemma is completed.
\end{proof}

From Lemma \ref{lem03}, we can obtain the following estimate about $\eta(t)$.

\begin{lemma}\label{lem04}
There exists a positive constant $C$
independent of $t$, such that for every $t\in[0,T)$,
\begin{equation*}
1/C \leq \eta(t) \leq C.
\end{equation*}
\end{lemma}

\begin{proof}
Applying the integration by substitution $\frac{h}{\mathcal{K}}\dd x=\rho^n\dd
u$, we have
\begin{equation*}
  \int_{\uS} G(\delbar h) h(x,t) /\mathcal{K} \,\dd x
  =\int_{\uS} G(\rho u) \rho(u,t)^n \dd u.
\end{equation*}
Inserting it into the definition of $\eta(t)$ given in \eqref{eta}, we obtain
\begin{equation*} 
  \eta(t) 
  = \frac{\displaystyle\int_{\uS} G(\rho u) \rho(u,t)^n \dd u}
  {\displaystyle\int_{\uS} f(x) h(x,t)/ \varphi(h) \dd x}.
\end{equation*}
Now the conclusion of this lemma follows directly from Lemma \ref{lem03}.
\end{proof}

Due to the convexity of $M_t$, Lemma \ref{lem03} also implies the gradient
estimates of $h(\cdot,t)$ and $\rho(\cdot,t)$. 

\begin{lemma}\label{lem05}
We have
\begin{gather*}
  |\nabla h(x,t)| \leq C,
  \quad \forall (x,t) \in \mathbb{S}^{n-1} \times [0, T), \\
  |\nabla \rho(u,t)| \leq C,
  \quad \forall (u,t) \in \mathbb{S}^{n-1} \times [0, T),
\end{gather*}
where $C$ is a positive constant depending only on the constant in Lemma \ref{lem03}.
\end{lemma}

\begin{proof}
By virtue of \eqref{eq:8}, there is
\begin{equation*}
  \rho^2=|\nabla h|^2+h^2,
\end{equation*}
which implies that
\begin{equation*}
|\nabla h|\leq \rho.
\end{equation*}
By \eqref{h}, \eqref{eq:8} and \eqref{eq:9}, we have
\begin{equation*}
h= \frac{\rho^2}{\sqrt{\rho^{2}+|\nabla \rho|^{2}}},
\end{equation*}
which implies that
\begin{equation*}
  |\nabla\rho|
  =\frac{|\nabla h|}{h}\rho
  \leq \frac{\rho^2}{h}.
\end{equation*}
The estimates of this lemma now follow directly from Lemma \ref{lem03}.
\end{proof}

\section{Long-time existence of the flow}

In this section, the long-time existence of the flow \eqref{flow}, or
equivalently \eqref{ht}, will be proved. 
In addition to the uniform bounds of support functions and their gradients
proved in the previous section,
we further need to establish uniform upper and lower bounds for
principal curvatures. 
These estimates can be obtained by considering proper auxiliary functions; see 
\cite{CHZ.MA.373-2019.953, LSW.JEMSJ.22-2020.893, LL.TAMS.373-2020.5833}
for similar techniques.

In the rest of this section, we take a local orthonormal frame $\{e_{1},
\cdots, e_{n-1}\}$ on $\mathbb{S}^{n-1} $ such that the standard metric on
$\mathbb{S}^{n-1} $ is $\{\delta_{ij}\}$, 
and double indices always mean to sum from $1$ to $n-1$.

By Lemma \ref{lem03}, along the flow \eqref{flow}, $h$ and $\delbar h$ always
range within a bounded interval $I'=[1/C,C]$ and a bounded annulus
$\Omega=\set{y\in\R^n : 1/C\leq|y|\leq C}$ respectively, where $C$ is the constant
in the lemma.

We first derive the uniform upper bound of the Gauss curvature of $M_t$.

\begin{lemma}\label{lem06}
There exists a positive constant $C$
independent of $t$, such that 
\begin{equation*}
  \mathcal{K}(x,t) \leq C,
  \quad \forall (x,t) \in \mathbb{S}^{n-1} \times [0, T).
\end{equation*}
\end{lemma}

\begin{proof}
Consider the following auxiliary function
\begin{equation*}
  W(x, t) =\frac{-\pd_th(x,t)+h(x,t)}{\eta(t)[h(x,t)-\varepsilon_0]},
\end{equation*}
where
\[ \varepsilon_0 =\frac{1}{2}\,\inf_{\uS\times[0,T)} h(x, t) \]
is a positive constant by virtue of Lemma \ref{lem03}.

For each $t\in[0,T)$, assume $W(\cdot,t)$ attains its maximum at some point
$x_t\in\uS$, namely
\begin{equation*}
W(x_t,t)=\max_{x\in\uS}W(x,t).
\end{equation*}
Then we have at $(x_t, t)$ that
\begin{equation}\label{K-1}
  0=\eta(t)W_i
  =\frac{-\pd_t h_i+h_i}{h-\varepsilon_0}
  +\frac{\pd_th- h}{(h-\varepsilon_0)^2}h_i,
\end{equation}
and
\begin{equation}\label{K-2}
  0\geq \eta(t)W_{ij}
  =\frac{-\pd_th_{ij}+h_{ij}}{h-\varepsilon_0}
  +\frac{(\pd_th-h)h_{ij}}{(h-\varepsilon_0)^2},
\end{equation}
where $W_i=0$ has been used in computing $W_{ij}$, and $W_{ij}\leq0$ should be
understood in the sense of negative semi-definite matrix.
As in the Introduction, write $b_{ij}=h_{ij}+h\delta_{ij}$, and $b^{ij}$ its
inverse matrix.
By \eqref{K-2}, we have
\begin{equation*}
  \begin{split}
    \pd_tb_{ij}
    &=\pd_th_{ij}+\pd_th\delta_{ij} \\
    &\geq h_{ij} +\frac{\pd_th-h}{h-\varepsilon_0}h_{ij} +\pd_th\delta_{ij} \\
    &= b_{ij} +\frac{\pd_th-h}{h-\varepsilon_0}h_{ij} +(\pd_th-h)\delta_{ij} \\
    &= b_{ij} +\frac{\pd_th-h}{h-\varepsilon_0}(h_{ij} +h\delta_{ij}-\varepsilon_0\delta_{ij}) \\
    &= b_{ij} -\eta(t) W (b_{ij} -\varepsilon_0\delta_{ij}).
\end{split}
\end{equation*}
Recalling $\mathcal{K}=1/\det(b_{ij})$, we obtain
\begin{equation}\label{eq:12}
  \begin{split}
    \pd_t\mathcal{K}
    &= -\mathcal{K}b^{ji}\pd_tb_{ij} \\
    &\leq -\mathcal{K}b^{ji}[b_{ij} -\eta(t) W (b_{ij} -\varepsilon_0\delta_{ij})] \\
    &= -\mathcal{K}\bigl[(n-1)(1 -\eta(t) W) +\eta(t) W\varepsilon_0 \TR(b^{ij})\bigr].
  \end{split}
\end{equation}

On the other hand, 
by \eqref{ht}, $W$ can be also written as
\begin{equation}\label{eq:10}
  W(x,t)
  = \frac{f(x)h}{(h-\varepsilon_0)\varphi(h) G(\delbar h)} \mathcal{K}(x,t).
\end{equation}
Then we have by Lemma \ref{lem03} that
\begin{equation} \label{eq:11}
  \frac{1}{C_1} W(x,t)\leq \mathcal{K}(x,t)\leq C_1 W(x,t),
\end{equation}
where $C_1$ is a positive constant depending only on the constant $C$ in Lemma
\ref{lem03}, and the upper and lower bounds of $f$ on $\uS$, $\varphi$ on $I'$,
and $G$ on $\Omega$. 
Noting that
\begin{equation*}
  \frac{1}{n-1}\TR(b^{ij})
  \geq \det(b^{ij})^{\frac{1}{n-1}}
  =\mathcal{K}^{\frac{1}{n-1}},
\end{equation*}
 and combining Lemma
\ref{lem04} and the inequalities \eqref{eq:12} and \eqref{eq:11}, we obtain 
\begin{equation}\label{eq:15}
  \begin{split}
    \pd_t\mathcal{K}
    &\leq (n-1) \eta(t) \mathcal{K} W -(n-1)\eta(t)\varepsilon_0 W \mathcal{K}^{\frac{n}{n-1}} \\
    &\leq C_2 W^2 -C_3 W^{\frac{2n-1}{n-1}},
  \end{split}
\end{equation}
where $C_2$ and $C_3$ are positive constants depending only on $n$,
$\varepsilon_0$, $C_1$, and the constant $C$ in Lemma \ref{lem04}.

Now we can estimate $\pd_tW$.
Let us begin with the following
\begin{equation}\label{eq:16}
  \begin{split}
    \frac{\pd}{\pd t} \Bigl[ & \frac{h}{(h-\varepsilon_0)\varphi(h) G(\delbar h)}  \Bigr] \\
    &= \frac{\pd_th}{(h-\varepsilon_0)\varphi G}
    -\frac{h\left( \pd_th\varphi G
        +(h-\varepsilon_0)\varphi'\pd_th G
        +(h-\varepsilon_0)\varphi \pd_t[G(\delbar h)] \right)}{[(h-\varepsilon_0)\varphi G]^2} \\
    &\leq C_4 \left( |\pd_th| +|\pd_t[G(\delbar h)]| \right),
  \end{split}
\end{equation}
where $C_4$ is a positive constant depending only on 
$\varepsilon_0$, the constant $C$ in Lemma \ref{lem03},
$\norm{\varphi}_{C^1(I')}$, and the lower bounds of
$\varphi$ on $I'$, and $G$ on $\Omega$.
To compute $\pd_t[G(\delbar h)]$,
we recall \eqref{K-1}, namely
\begin{equation*}
  \pd_t h_i =(1 -\eta W) h_i.
\end{equation*}
Then
\begin{equation*}
  \begin{split}
    \pd_t\delbar h
    &= \pd_t \Bigl( \sum_i h_ie_i+hx \Bigr) =\sum_i \pd_t h_ie_i+\pd_th x \\
    &= (1 -\eta W) \sum_i h_ie_i+[(1-\eta W)h+\eta W\varepsilon_0] x \\
    &= (1 -\eta W) \delbar h +\eta W\varepsilon_0 x.
  \end{split}
\end{equation*}
Therefore,
\begin{equation*} 
  \begin{split}
    \pd_t[G(\delbar h)]
    &= \langle \delbar G,\pd_t \delbar h \rangle \\
    &= (1 -\eta W) \langle \delbar G,\delbar h \rangle
    + \eta W\varepsilon_0 \langle \delbar G,x \rangle.
  \end{split}
\end{equation*}
Now recalling the definition of $W$, we have
\begin{equation*}
  \begin{split}
    |\pd_th| &+|\pd_t[G(\delbar h)]| \\
    &\leq |(1-\eta W)h+\eta W\varepsilon_0|
    + |1 -\eta W| \cdot |\delbar G| \cdot |\delbar h|
    + \eta W\varepsilon_0 |\delbar G| \\
    &\leq C_5(1+W),
  \end{split}
\end{equation*}
where $C_5$ is a positive constant depending only on 
the constants in Lemmas \ref{lem03} and \ref{lem04}, and
$\norm{G}_{C^1(\Omega)}$.
Inserting it into \eqref{eq:16},
we obtain that
\begin{equation*}
  \frac{\pd}{\pd t} \Bigl[ \frac{h}{(h-\varepsilon_0)\varphi(h) G(\delbar h)} \Bigr]
  \leq C_4C_5(1+W).
\end{equation*}
By virtue of \eqref{eq:10}, and note \eqref{eq:15}, at $(x_t,t)$, there is
\begin{equation}\label{eq:14}
  \begin{split}
    \pd_tW
    &= \frac{\pd}{\pd t}\Bigl[ \frac{h}{(h-\varepsilon_0)\varphi(h) G(\delbar h)} \Bigr]f\mathcal{K}
    + \frac{hf\pd_t\mathcal{K}}{(h-\varepsilon_0)\varphi(h) G(\delbar h)} \\
    &\leq C_4C_5(1+W)f\mathcal{K}
    + \frac{W}{\mathcal{K}}
    \bigl(
    C_2 W^2 -C_3 W^{\frac{2n-1}{n-1}}
    \bigr) \\
    &\leq C_4C_5f_{\max}C_1(W+W^2)
    +C_1C_2W^2
    -C_1^{-1}C_3 W^{\frac{2n-1}{n-1}} \\
    &= W^{2} \bigl[ C_4C_5f_{\max}C_1(W^{-1}+1)
    +C_1C_2
    -C_1^{-1}C_3 W^{\frac{1}{n-1}}\bigr],
  \end{split}
\end{equation}
where we have used \eqref{eq:11} to obtain the second inequality.
Now one can see that whenever $W(x_t,t)$ is greater than some constant which is
independent of $t$,
\begin{equation*}
\pd_tW<0,
\end{equation*}
which implies that $W$ has a uniform upper bound.
By \eqref{eq:11}, $\mathcal{K}$ has a uniform upper bound.
\end{proof}

Now we can estimate lower bounds of principal curvatures $\kappa_{i}(x,t)$ of $M_t$
for $i=1,\cdots, n-1$.

\begin{lemma}\label{lem07}
  We have
  \begin{equation*}
  \kappa_{i}(x,t) \geq C, \quad \forall (x,t)\in\uS\times[0,T),
  \end{equation*}
  where $C$ is a positive constant independent of $t$.
\end{lemma}

\begin{proof}
We consider the auxiliary function
\begin{equation*}
  \widetilde{\Lambda}(x, t)
  =\log \lambda_{\max}(b_{ij})-A\log h+B |\nabla h|^2,
  \quad \forall (x,t)\in\uS\times[0,T),
\end{equation*}
where $b_{ij}=h_{ij}+h\delta_{ij}$ as before, $\lambda_{\max}(b_{ij})$ denotes
the maximal eigenvalue of the matrix $(b_{ij})$, and $A$ and $B$ are positive
constants to be chosen later. 
If we can prove that $\widetilde{\Lambda}$ has a uniform upper bound, then by
Lemmas \ref{lem03} and \ref{lem05}, $\lambda_{\max}(b_{ij})$ has a uniform upper
bound, which further implies the conclusion of this lemma.

For any fixed $T'\in(0,T)$, there exists some point
$(x_0,t_0)\in\uS\times[0,T']$ such that
\begin{equation}\label{eq:17}
\widetilde{\Lambda}(x_0,t_0) =\max_{\uS\times[0,T']} \widetilde{\Lambda}(x,t).
\end{equation}
By choosing a suitable orthonormal frame, we may assume $(b_{ij}(x_0, t_0))$ is
diagonal and $\lambda_{\max}(b_{ij})(x_0, t_0)=b_{11}(x_0, t_0)$.
Now we consider the following new auxiliary function
\begin{equation*}
  \Lambda(x, t)=\log b_{11}-A\log h+B |\nabla h|^2,
  \quad \forall (x,t)\in\uS\times[0,T'].
\end{equation*}
Since $\Lambda(x_0,t_0)=\widetilde{\Lambda}(x_0,t_0)$, $\Lambda(x,t)$ attains its
maximum at $(x_0,t_0)$.
At this point, we have
\begin{equation}\label{C2-1d}
0=\Lambda_i=b^{11}b_{11; i}-A\frac{h_i}{h}+2B \sum_{k}h_k h_{ki},
\end{equation}
and
\begin{equation} \label{C2-2d}
  \begin{split}
    0\geq \Lambda_{ij}
    &=b^{11}b_{11; ij}-(b^{11})^2 b_{11; i}b_{11; j}
    -A\Bigl(\frac{h_{ij}}{h}-\frac{h_i h_j}{h^2}\Bigr) \\
    &\hskip1.1em +2B \sum_{k}(h_{kj} h_{ki}+h_kh_{kij}),
  \end{split}
\end{equation}
where $(b^{ij})$ denotes the inverse of the matrix $(b_{ij})$, and
$\Lambda_{ij}\leq0$ means that it is a negative semi-definite matrix.
Without loss of generality, we can assume $t_0>0$.
Then at $(x_0,t_0)$, we also have
\begin{equation} \label{eq:13}
  \begin{split}
  0\leq\pd_t\Lambda
  &=b^{11}\pd_tb_{11} -A\frac{\pd_th}{h} +2B\sum_k h_k\pd_th_k \\
  &=b^{11}(\pd_th_{11}+\pd_th) -A\frac{\pd_th}{h} +2B\sum_k h_k\pd_th_k.
  \end{split}
\end{equation}

We rewrite the flow \eqref{ht} as
\begin{equation*}
 \log(h- \pd_th)
  = \log \mathcal{K}(x,t) +\log\eta(t) + \log\Bigl[ \frac{f(x) h}{\varphi(h) G(\delbar h)}\Bigr].
\end{equation*}
Denoting
\begin{equation*}
  \alpha(x,t)
  = \log\Bigl[ \frac{f(x) h}{\varphi(h) G(\delbar h)}\Bigr],
\end{equation*}
and differentiating the both sides, we have at $(x_0,t_0)$ that
\begin{equation}\label{C2-00}
\frac{h_k-\pd_th_k}{h-\pd_th}=-b^{ji}b_{ij; k}+\alpha_k,
\end{equation}
and
\begin{equation}\label{C2-11}
  \frac{h_{11}-\pd_th_{11}}{h-\pd_th}
  =\frac{(h_{1}-\pd_th_{1})^2}{(h-\pd_th)^2}
  -b^{ii}b_{ii; 11} +b^{ii}b^{jj}(b_{ij; 1})^2 +\alpha_{11}.
\end{equation}
By the Ricci identity, there is
\begin{equation*}
b_{ii; 11}=b_{11; ii}-b_{11}+b_{ii}.
\end{equation*}
Multiplying both sides of \eqref{C2-11} by $-b^{11}$, we obtain
\begin{equation}\label{eq:19}
  \begin{split}
  \frac{b^{11}\pd_th_{11}-b^{11}h_{11}}{h-\pd_th}
  &\leq b^{11}b^{ii}b_{ii; 11} -b^{11}b^{ii}b^{jj}(b_{ij; 1})^2 -b^{11}\alpha_{11} \\
  &\leq b^{11}b^{ii}b_{11; ii} -b^{11}b^{ii}b^{11}(b_{i1; 1})^2 -\sum_i b^{ii} \\
  &\hskip1.1em +b^{11}(n-1-\alpha_{11}).
  \end{split}
\end{equation}
From \eqref{C2-2d}, there is $b^{ij}\Lambda_{ij}\leq0$, namely
\begin{equation*}
  \begin{split}
    b^{11}b^{ii}b_{11; ii}-(b^{11})^2b^{ii} (b_{11; i})^2
    \leq
    Ab^{ii}\Bigl(\frac{h_{ii}}{h}-\frac{h_i^2}{h^2}\Bigr)
    -2B \sum_{k}b^{ii}(h_{ki}^2+h_kh_{kii}).
  \end{split}
\end{equation*}
Note that
\begin{gather*}
  b^{ii}h_{ii}
  =b^{ii}(b_{ii}-h)
  =n-1-h\sum_i b^{ii}, \\
  \sum_{k}b^{ii}h_{ki}^2
  =b^{ii}h_{ii}^2
  =b^{ii}(b_{ii}^2-2hb_{ii}+h^2)
  = -2(n-1)h+\sum_i b_{ii} +h^2\sum_i b^{ii},  \\
  \sum_{k}b^{ii}h_kh_{kii}
  =\sum_{k}b^{ii}h_k(b_{ki;i}-h_i\delta_{ki})
  =\sum_kb^{ii}h_kb_{ii;k} -b^{ii}h_i^2,
\end{gather*}
where the fact that $b_{ij;k}$ is symmetric in all indices is used, therefore
\begin{equation*}
  \begin{split}
    b^{11}b^{ii}b_{11; ii}-(b^{11})^2b^{ii} (b_{11; i})^2
    &\leq
    \frac{(n-1)A}{h} -A \sum_i b^{ii}
    -\frac{Ab^{ii}h_i^2}{h^2} \\
    &\hskip1.2em  +4(n-1)Bh -2B\sum_i b_{ii} -2Bh^2\sum_i b^{ii} \\
    &\hskip1.2em -2B\sum_k b^{ii}h_kb_{ii;k} +2Bb^{ii}h_i^2.
  \end{split}
\end{equation*}
Inserting it into \eqref{eq:19}, we obtain that
\begin{equation}\label{eq:20}
  \begin{split}
    \frac{b^{11}\pd_th_{11}-b^{11}h_{11}}{h-\pd_th}
    &\leq
    \frac{(n-1)A}{h} -(A+2Bh^2+1) \sum_i b^{ii}
    -\frac{A-2Bh^2}{h^2}b^{ii}h_i^2 \\
    &\hskip1.2em  +4(n-1)Bh -2B\sum_i b_{ii} \\
    &\hskip1.2em -2B\sum_k b^{ii}h_kb_{ii;k} +b^{11}(n-1-\alpha_{11}).
  \end{split}
\end{equation}
Recalling \eqref{C2-00}, we have
\begin{equation}\label{eq:18}
\frac{2B\sum_k h_k\pd_th_k}{h-\pd_th}
= \frac{2B|\nabla h|^2}{h-\pd_th} +2B\sum_kb^{ii}h_kb_{ii;k}
-2B\langle \nabla h,\nabla\alpha \rangle.
\end{equation}
Now dividing \eqref{eq:13} by $h-\pd_th$, there is
\begin{equation*}
  \begin{split}
    0 &\leq \frac{b^{11}(\pd_th_{11}-h_{11}+b_{11}-h+\pd_th)}{h-\pd_th}
    -\frac{A\pd_th}{h(h-\pd_th)}
    +\frac{2B\sum_k h_k\pd_th_k}{h-\pd_th}  \\
    &= \frac{b^{11}(\pd_th_{11}-h_{11})}{h-\pd_th}
    +\frac{2B\sum_k h_k\pd_th_k}{h-\pd_th} 
    -b^{11} +\frac{A}{h}-\frac{A-1}{h-\pd_th},
  \end{split}
\end{equation*}
which together with \eqref{eq:20} and \eqref{eq:18} implies that
\begin{equation*}
  \begin{split}
    0
    &\leq \frac{nA}{h} -(A+2Bh^2+1) \sum_i b^{ii}
    -\frac{A-2Bh^2}{h^2}\sum_i b^{ii}h_i^2 \\
    &\hskip1.2em  +4(n-1)Bh -2B\sum_i b_{ii}  +b^{11}(n-2-\alpha_{11}) \\
    &\hskip1.2em -2B\langle \nabla h,\nabla\alpha \rangle 
    -\frac{A-1-2B|\nabla h|^2}{h-\pd_th}.
  \end{split}
\end{equation*}
If we choose $A=n+2BC^2$, where $C$ is the constant in Lemma \ref{lem03}, then
there is
\begin{equation}\label{eq:22}
  (A-n+3) \sum_i b^{ii}
  +2B\sum_i b_{ii} 
  \leq
  C_1(A+B)
  -b^{11}\alpha_{11}
  -2B\langle \nabla h,\nabla\alpha \rangle,
\end{equation}
where $C_1$ is a positive constant depending only on $n$ and the constant $C$
in Lemma \ref{lem03}.

Recall
\begin{equation*}
  \alpha(x,t)
  =\log f +\log h -\log\varphi(h) -\log G(\delbar h),
\end{equation*}
and note that
\begin{equation*}
  \nabla_i [G(\delbar h)]
  =\langle \delbar G, \delbar_i \delbar h \rangle
  =\sum_k\langle \delbar G, e_k \rangle b_{ki},
\end{equation*}
and
\begin{equation*}
  \begin{split}
    \nabla_{ji}^2 [G(\delbar h)]
    &=\langle \langle \delbar^2 G,\delbar_j\delbar h \rangle, \delbar_i\delbar h \rangle
    +\langle \delbar G, \delbar_{ji}^2 \delbar h \rangle \\
    &=\sum_{k,l}\langle \langle \delbar^2 G,e_l \rangle, e_k \rangle b_{lj}b_{ki}
    -\langle \delbar G, x \rangle b_{ji} +\sum_k\langle \delbar G, e_k \rangle b_{ki;j}.
  \end{split}
\end{equation*}
Then we have 
\begin{equation*}
  \begin{split}
    -2B\langle \nabla h,\nabla\alpha \rangle 
    &= -2B \sum_k h_k \left(
      \frac{f_k}{f}
      +\frac{h_k}{h}
      -\frac{\varphi'(h)h_k}{\varphi(h)}
      -\frac{\langle \delbar G, e_l \rangle b_{lk}}{G(\delbar h)} 
    \right) \\
    &\leq C_2B 
    +\sum_k \frac{2Bh_k\langle \delbar G, e_k \rangle h_{kk}}{G(\delbar h)},
  \end{split}
\end{equation*}
where $C_2$ is a positive constant depending only on the constants $C$ in Lemmas
\ref{lem03} and \ref{lem05}, $\norm{f}_{C^1(\uS)}$, $\norm{\varphi}_{C^1(I')}$,
$\norm{G}_{C^1(\Omega)}$, $\min\limits_{\uS} f$, $\min\limits_{I'}\varphi$, and
$\min\limits_{\Omega} G$. 
We also have
\begin{equation*}
  \begin{split}
    -\alpha_{11} 
    &= \frac{f_1^2}{f^2} -\frac{f_{11}}{f}
    +\frac{h_1^2}{h^2}  -\frac{h_{11}}{h} 
    +\frac{\varphi''h_1^2+\varphi'h_{11}}{\varphi}
    -\frac{(\varphi')^2h_1^2}{\varphi^2}
    -\frac{[\langle \delbar G, e_1 \rangle b_{11}]^2}{G^2}\\
    &\hskip1.2em +\frac{1}{G} \Bigl[
    \langle \langle \delbar^2 G,e_1 \rangle, e_1 \rangle b_{11}^2
    -\langle \delbar G, x \rangle b_{11}
    +\sum_k\langle \delbar G, e_k \rangle b_{k1;1} \Bigr] \\
    &\leq C_3(1+b_{11}+b_{11}^2) +\frac{\sum_k\langle \delbar G, e_k \rangle b_{11;k}}{G},
  \end{split}
\end{equation*}
where $C_3$ is a positive constant depending only on the constants $C$ in Lemmas
\ref{lem03} and \ref{lem05}, $\norm{f}_{C^2(\uS)}$, $\norm{\varphi}_{C^2(I')}$,
$\norm{G}_{C^2(\Omega)}$, $\min\limits_{\uS} f$, $\min\limits_{I'}\varphi$, and
$\min\limits_{\Omega} G$. 
Therefore,
\begin{equation}\label{eq:21}
  \begin{split}
    -b^{11}\alpha_{11} -2B\langle \nabla h,\nabla\alpha \rangle 
    &\leq C_3(b^{11}+1+b_{11}) + C_2B \\
    &\hskip1.1em +\sum_k\frac{\langle \delbar G, e_k \rangle}{G} (b^{11}b_{11;k} + 2Bh_k h_{kk}) \\
    &= C_3(b^{11}+1+b_{11}) + C_2B \\
    &\hskip1.1em +\sum_k\frac{\langle \delbar G, e_k \rangle}{G}\cdot \frac{Ah_k}{h} \\
    &\leq C_3(b^{11}+1+b_{11}) + C_2B +C_4A,
  \end{split}
\end{equation}
where we have used the equality \eqref{C2-1d}, and $C_4$ is a positive constant
depending only on the constants $C$ in Lemmas \ref{lem03} and \ref{lem05},
$\norm{G}_{C^1(\Omega)}$, and $\min\limits_{\Omega} G$.  
Inserting \eqref{eq:21} into \eqref{eq:22}, we obtain that
\begin{equation*}
  (A-n) \sum_i b^{ii}
  +2B\sum_i b_{ii} 
  \leq
  (C_1+C_4)A
  +(C_1+C_2)B
  +C_3(b^{11}+1+b_{11}),
\end{equation*}
which together with $A=n+2BC^2$ implies that
\begin{equation*}
  \begin{split}
  (2BC^2-C_3) \sum_i b^{ii}
  &+(2B-C_3)\sum_i b_{ii} \\
  &\leq
  (C_1+C_4)(n+2BC^2)
  +(C_1+C_2)B
  +C_3.
  \end{split}
\end{equation*}
If we let $B=\max\set{\frac{C_3+1}{2}, \frac{C_3+1}{2C^2}}$, we see that
$b_{11}(x_0,t_0)$ is bounded from above by a positive constant depending only on
$n$, $C_1$, $C_2$, $C_3$ and $C_4$.
By the choice of $(x_0,t_0)$, and recalling Lemmas \ref{lem03} and
\ref{lem05}, $\max\limits_{\uS\times[0,T']} \widetilde{\Lambda}(x,t)$ is bounded from
above by a positive constant depending only on $n$, $C$, $C_1$, $C_2$, $C_3$ and
$C_4$. Since $T'$ can be any number in $(0,T)$,
$\sup\limits_{\uS\times[0,T)} \widetilde{\Lambda}(x,t)$ is bounded from above by a
positive constant independent of $T$.
The proof of this lemma is completed.
\end{proof}

Combining Lemma \ref{lem06} and Lemma \ref{lem07}, we see that
the principal curvatures of $M_{t}$ has uniform positive upper and lower bounds.
This together with Lemmas \ref{lem03}, \ref{lem04} and \ref{lem05} implies that the
evolution equation \eqref{ht} is uniformly parabolic on any finite time
interval. Thus, the result of \cite{KS.IANSSM.44-1980.161} and the standard
parabolic theory show that the smooth solution of \eqref{ht} exists for all
time.
Moreover, by these estimates again, a subsequence of $M_t$ converges in $C^\infty$ to
a positive, smooth, uniformly convex hypersurface $M_\infty$ in $\R^n$.
Now to complete the proof of Theorem \ref{thm2}, it remains to check the support
function of $M_\infty$ satisfies Eq. \eqref{dOMP-f} for some positive constant $c$.

\section{Existence of solutions}

In this section, we first complete the proof of Theorem \ref{thm2}.
Let $\tilde{h}$ be the support function of $M_\infty$. We need to prove that
$\tilde{h}$ is a solution to the following equation
\begin{equation} \label{dOMP-f1}
  c\, \varphi(h) G(\delbar h) \det(\nabla^2h +hI) =f \text{ on } \uS
\end{equation}
for some positive constant $c$.
This will be achieved by analyzing the functional $J(t)$ defined in \eqref{Jt}.

In fact, 
when the condition \eqref{cond1} holds, by Lemma \ref{lem02}, $J'(t)\leq0$ for
any $t>0$. 
From
\begin{equation*}
\int_0^t [-J'(t)] \dd t =J(0)-J(t) \leq J(0),
\end{equation*}
we have
\begin{equation*}
\int_0^\infty [-J'(t)] \dd t  \leq J(0).
\end{equation*}
This implies that there exists a subsequence of times $t_j\to\infty$ such that
\begin{equation*}
-J'(t_j) \to 0 \text{ as } t_j\to\infty.
\end{equation*}
Recalling \eqref{eq:23}:
\begin{equation*}
  \begin{split}
    J'&(t_j) \int_{\uS} f h/ \varphi(h) \dd x \\
    &= \left(
      \int_{\uS}
      \sqrt{\frac{Gh}{\mathcal{K}}}
      \cdot
      \frac{f}{\varphi}
      \sqrt{\frac{h\mathcal{K}}{G}}
      \dd x
    \right)^2
    -\int_{\uS} \frac{G h}{\mathcal{K}} \dd x
    \int_{\uS}\frac{f^2h\mathcal{K}}{\varphi^2G} \dd x.
  \end{split}
\end{equation*}
Since $h$, $|\delbar h|$ and $\mathcal{K}$ have uniform positive upper and lower bounds, by
passing to the limit, we obtain
\begin{multline*}
    \biggl(
      \int_{\uS}
      \sqrt{\frac{G(\delbar\tilde{h})\tilde{h}}{\widetilde{\mathcal{K}}}}
      \cdot
      \frac{f}{\varphi(\tilde{h})}
      \sqrt{\frac{\tilde{h}\widetilde{\mathcal{K}}}{G(\delbar\tilde{h})}}
      \dd x
    \biggr)^2 \\
    =\int_{\uS} \frac{G(\delbar\tilde{h}) \tilde{h}}{\widetilde{\mathcal{K}}} \dd x
    \int_{\uS}\frac{f^2\tilde{h}\widetilde{\mathcal{K}}}{\varphi(\tilde{h})^2G(\delbar\tilde{h})} \dd x,
\end{multline*}
where $\widetilde{\mathcal{K}}$ is the Gauss curvature of $M_\infty$.
By the equality condition for the H\"older's inequality, there exists a constant
$c\geq0$ such that
\begin{equation*}
  c^2\, \frac{G(\delbar\tilde{h}) \tilde{h}}{\widetilde{\mathcal{K}}} 
  = \frac{f^2\tilde{h}\widetilde{\mathcal{K}}}{\varphi(\tilde{h})^2G(\delbar\tilde{h})} 
  \ \text{ on }\uS,
\end{equation*}
namely
\begin{equation*}
  \frac{c\, \varphi(\tilde{h}) G(\delbar\tilde{h})}{\widetilde{\mathcal{K}}} =f
  \ \text{ on }\uS,
\end{equation*}
which is just equation \eqref{dOMP-f1}.
Note $\tilde{h}$, $|\delbar\tilde{h}|$ and $\widetilde{\mathcal{K}}$ have
positive upper and lower bounds, $c$ should be positive.

For the proof when the condition \eqref{cond2} holds.
Noting Lemma \ref{lem03}, there exists a positive constant $C$ which is
independent of $t$, such that $J(t)\leq C$ for any $t\geq0$.
Recalling Lemma \ref{lem02}, $J'(t)\geq0$ for any $t>0$.
Then
\begin{equation*}
\int_0^t J'(t) \dd t =J(t)-J(0) \leq J(t) \leq C,
\end{equation*}
which leads to
\begin{equation*}
\int_0^\infty J'(t) \dd t  \leq C.
\end{equation*}
This implies that there exists a subsequence of times $t_j\to\infty$ such that
\begin{equation*}
J'(t_j) \to 0 \text{ as } t_j\to\infty.
\end{equation*}
Now using almost the same arguments as above, one can prove $\tilde{h}$ solves
Eq. \eqref{dOMP-f1} for some positive constant $c$.
The proof of Theorem \ref{thm2} is completed.
Meanwhile, Theorem \ref{thm1} is a direct consequence of Theorem \ref{thm2}.


\end{document}